\documentclass[a4paper,11pt,twoside,leqno]{article}
\usepackage{amsmath,amssymb,amsfonts,amsthm,graphicx,fancyhdr}
\usepackage[latin1]{inputenc}
\usepackage[T1]{fontenc}
\usepackage{rotating}
\usepackage[colorlinks=true]{hyperref}
\usepackage{thmtools}

\newtheorem{teo}{Theorem}[section]
\newtheorem{lem}[teo]{Lemma}
\newtheorem{prop}[teo]{Proposition}
\newtheorem{cor}[teo]{Corollary}

\newtheorem{ques}[teo]{Question}

\declaretheoremstyle[
  spaceabove=\topsep, spacebelow=\topsep,
  headfont=\bf,  
  notefont=\mdseries, notebraces={(}{)},
  bodyfont=\rmfamily, 
  postheadspace=1em,
  qed=$\Diamond$
]{drem}
\declaretheorem[style=drem, name=Remark, numberlike=teo]{rmk}

\newcommand{\eg}[0]{\emph{e.g.} }
\newcommand{\ie}[0]{\emph{i.e.} }

\newcommand{\srl}[1]{\overline{#1}}
\newcommand{\jo}[1]{\mathcal{#1}}

\DeclareFontFamily{T1}{mafra}{}
\DeclareFontShape{T1}{mafra}{m}{n}{<->s*[0.95]yswab}{} 
\DeclareFontShape{T1}{mafra}{m}{it}{<->s*[1.0]ygoth}{} 
\DeclareTextFontCommand{\textgoth}{\yfrak}

\DeclareSymbolFont{mafrak}{T1}{mafra}{m}{n}
\DeclareSymbolFontAlphabet{\mathfr}{mafrak}

\DeclareSymbolFont{mbbold}{U}{bbold}{m}{n}
\DeclareSymbolFontAlphabet{\mathbbold}{mbbold}

\newcommand{\mr}[1]{\mathrm{#1}}
\newcommand{\mb}[1]{\mathbbold{#1}}
\newcommand{\mc}[1]{\mathcal{#1}}
\newcommand{\ms}[1]{\mathsf{#1}}
\newcommand{\mf}[1]{\mathfr{#1}}

\newcommand{\pgen}[1]{\langle #1 \rangle}

\newcommand{\surj}[0]{\twoheadrightarrow}

\newcommand{\eps}[0]{\varepsilon}

\newcommand{\rr}[0]{\ensuremath{\mathbb{R}}}
\newcommand{\zz}[0]{\ensuremath{\mathbb{Z}}}
\newcommand{\nn}[0]{\ensuremath{\mathbb{N}}}

\newcommand{\del}[0]{\ensuremath{\partial}}

\newcommand{\img}[0]{\mathrm{Im}\,}

\newcommand{\un}[0]{\mathbbold{1}}

\newcommand{\inff}[1]{\textrm{\raisebox{.5ex}{\mbox{$\underset{#1}{\inf}$}}} \:}

\newcommand{\comp}[0]{\mathsf{c}}
\newcommand{\perpud}[0]{{\, \textrm{\mbox{\raisebox{.7ex}{\:\begin{rotate}{180}\makebox(0,0){$\perp$}\end{rotate}}}} \,}}

\begin{document}

\renewcommand{\thefootnote}{\fnsymbol{footnote}}

\renewcommand{\thefootnote}{\arabic{footnote}}

\newcommand{\ud}{\tfrac{1}{2}}
\newcommand{\ut}{\tfrac{1}{3}}
\newcommand{\uq}{\tfrac{1}{4}}

\newcommand{\lmd}{\lambda}
\newcommand{\Lmd}{\Lambda}
\newcommand{\Gm}{\Gamma}
\newcommand{\gm}{\gamma}
\newcommand{\GM}{\Gamma}

\newcommand{\bsl}\backslash
\newcommand{\acts}{\curvearrowright}
\newcommand{\donc}{\rightsquigarrow}

\newcommand{\smdd}[4]{\big( \begin{smallmatrix}#1 & #2 \\ #3 & #4\end{smallmatrix} \big)}

\newcommand{\sgn}{\textrm{sgn}\,}

\begin{center}
\Large Cuts, flows and gradient conditions on harmonic functions.
\vspace*{1cm}

\centerline{\large Antoine Gournay\thanks{This work was supported by the ERC Consolidator Grant No. 681207, ``Groups, Dynamics, and Approximation''.}}
\end{center}

\vspace*{1cm}

\centerline{\textsc{Abstract}}

\begin{center}
\parbox{10cm}{{ \small 
Reduced cohomology motivates to look at harmonic functions which satisfy certain gradient conditions. If $G$ is a direct product of two infinite groups or a (FC-central)-by-cyclic group, then there are no harmonic functions with gradient in $c_0$ on its Cayley graphs. From this, it follows that a metabelian group $G$ has no harmonic functions with gradient in $\ell^p$.

\hspace*{.1ex} 
}}
\end{center}

\section{Introduction}\label{s-intro}

\newcommand{\hdc}[0]{t\!\mc{H}\!\ms{D}^c}
\newcommand{\hdp}[0]{t\!\mc{H}\!\ms{D}^p}
\newcommand{\hdq}[0]{t\!\mc{H}\!\ms{D}^q}


The subject matter of this paper is to investigate which graphs and Cayley graph of groups possess harmonic functions whose gradient (the difference of the values of the function at the end of the edges) belongs to $\ell^p$ or $c_0$. The motivation comes mainly from groups: for example when $p=\infty$, one gets the class of Lipschitz harmonic functions which is of known importance, \eg see Shalom \& Tao \cite{ST}. 

Furthermore, if a Hilbertian representation of the group has non trivial reduced cohomology in degree 1, then there exist a non-constant harmonic function on the Cayley graph.
The gradient of this harmonic function is related to the mixing property of this representation. 
For example, for a strongly mixing representation this yields (in any Cayley graph) a harmonic functions with gradient in $c_0$. 
Hence if a group has no harmonic function with gradient in $c_0$ in some Cayley graph, then the reduced cohomology in degree 1 of any strongly mixing representation is trivial; see \cite[\S{}2]{GJ} or \cite[\S{}3]{Go-mixing} for details and references. 

Lastly, let us mention the reduced $\ell^p$-cohomology in degree 1 (of a group or graph), an useful quasi-isometry invariant. 
Under some assumptions on the isoperimetry, the non-vanishing of this cohomology is equivalent to the presence of harmonic function with gradient in $\ell^p$; see \cite{Go-frontier} for details and references. An underlying question due to Gromov \cite[\S{}8.$A_1$.$A_2$, p.226]{Gro} is whether any amenable group has harmonic function with gradient in $\ell^p$.

For Cayley graphs of groups, the main results are:
\begin{teo}[see Theorem \ref{tc0abcy} and Corollary \ref{tlpmetab-c}]\label{teo1}
Let $G$ be a finitely generated group which is a cyclic extension of a group with an infinite FC-centraliser.
Then (on any Cayley graph of $G$) any harmonic function with gradient in $c_0$ is constant.

Consequently, on any Cayley graph of a metabelian group there are no harmonic function with gradient in $\ell^p$.
\end{teo}
Note that any virtually nilpotent group has an infinite FC-centraliser (the FC-centraliser is the subgroup of elements with a finite conjugacy class; more see \S{}\ref{sscentmal}). The previous result is akin to a result of Brieussel \& Zheng \cite[Theorem 1.1]{BZ}. 

The current technique sometime only apply to a restricted generating set:
\begin{teo}[see Theorem \ref{tsubcent}]\label{teo2}
Let $G$ be a group and $H$ an infinite subgroup so that $G$ is generated by $H$ and its FC-centraliser $Z:= Z^{FC}_G(H)$. Then there is a Cayley graph of $G$ which has no non-constant harmonic function with gradient in $c_0$.
\end{teo}
Note that if $Z \cap H$ is infinite, then the FC-centraliser of $G$ itself is infinite. 
In that case \cite[Lemma 2.7]{GJ} could already yield the conclusion; see also Corollary \ref{tzinf1-c}.

Direct products of two infinite groups satisfy the hypothesis of Theorem \ref{teo2}. For such groups, the absence of harmonic functions with gradient in $c_0$ can be established in any Cayley graph; see Corollary \ref{tdirc0gp-c}. The results mention so far rely on the concepts of almost-malnormal subgroups and quasi-normalisers; in that sense, they build up on some techniques from \cite{Go-mixing}. 

The article is organised as follows. All of the questions are gathered in  \S{}\ref{sques}.
\S{}\ref{sharmetcoho} introduces the space of cuts and flows which are used to study the space of harmonic function with gradient conditions. 
One of the aim of this section is to give a dual characterisation of having no non-constant harmonic function with gradient in $\ell^p$ (or $c_0$). 
This dual characterisation will allow to reduce the non-existence of harmonic functions to the existence of transport plans to deal with other cases.
\S{}\ref{ssliou} introduces the idea of exit distributions which also come in the construction of some transport plans and \S{}\ref{ssisop} defines isoperimetric profiles which are a key condition for some results.
\S{}\ref{shdp} contains the bulk of the proofs. 
\S{}\ref{sslioudir} defines transport plans and use them to give alternate proofs of results from \cite{Go}, namely that graph without bounded harmonic functions also have no [potentially unbounded] harmonic function with gradient in $\ell^p$, provided some isoperimetric inequality holds; this uses the exit distributions from \S{}\ref{ssliou}.
\S{}\ref{ss-c0centr} reproves \cite[Lemma 2.7]{GJ} using slightly different techniques. 
\S{}\ref{sscentmal} hinges on the idea of q-normalisers to show that the constancy of a harmonic function with gradient in $c_0$ propagates to malnormal hulls. This eventually lead to the proof of Theorems \ref{teo1} and \ref{teo2}.

{\it Acknowledgments:} The author learned about Theorem \ref{tliouex-t} through an answer V.~Beffara gave on MathOverflow \cite{Beff-MO} ans was given the correct reference thanks to V.~Kaimanovich.

\section{Harmonic functions and reduced $\ell^p$-cohomology}\label{sharmetcoho}

\subsection{Cuts and flows}\label{sscutflow}

The aim of this section is to define the subspaces of $\ell^pE$ and $c_0E$ which are of interest here (cuts and flows) and relate them to harmonic functions.

For convenience, the edges will be seen as a subset of $V \times V$. However to avoid dealing with ``alternate'' functions\footnote{When both $(x,y)$ and $(y,x)$ are allowed as edges, the gradient of a function is alternated in the sense it satisfies $f(x,y) = -f(y,x)$. Since alternated functions in $\ell^pE$ are complemented, this is not a serious difficulty, simply a convention.}, 
only one of the pair $(x,y)$ or $(y,x)$ will be allowed. So $x,y \in V$ are neighbours if either $(x,y) \in E$ or $(y,x) \in E$; but $(x,y)$ and $(y,x)$ may not both belong to $E$.

That said, given a function on the vertices $f:V \to \rr$, $\nabla f$ is defined as $\nabla f(x,y) = f(y)-f(x)$ (whenever $(x,y) \in E$). On any countable set $X$, one can define the inner product of finitely supported functions $f,g:X \to \rr$ by
\[
\pgen{f \mid g} = \sum_{x \in X} f(x) g(x).
\]
This extends to larger spaces, \eg $f \in \ell^pX$ and $g \in \ell^{p'}X$. The adjoint of $\nabla$ for this product is $\nabla^*$. It is defined (as usual) by $\pgen{ \nabla^* f \mid g} = \pgen{f \mid \nabla g}$ where $f:E \to \rr$ and $g: V \to \rr$. This gives
\[
\nabla^* f(x) = - \sum_{ (x,y) \in E} f(x,y) + \sum_{ (y,x) \in E} f(y,x).
\]
A function $f:V \to \rr$ is harmonic if $\nabla^* \nabla f \equiv 0$. 

The standing hypothesis that the graph is connected is important in order that the only function with trivial gradient are constant functions. 
The [other] standing hypothesis that the graph has bounded valency is crucial in order for the gradient to be a bounded operator (from $\ell^pV \to \ell^pE$). 
Note that the identity $\pgen{ \nabla^* f \mid g} = \pgen{f \mid \nabla g}$ holds if $f \in \ell^p E$ and $g \in \ell^{p'}V$ (where $p'$ is the H\"older conjugate of $p$ and $\ell^\infty$ can be replaced by $c_0$).

The notation $\ell^0X$ will be used to speak of the finitely supported functions on $X$. The space $c_0X$ is the completion of $\ell^0X$ with respect to the $\ell^\infty$-norm. A function $f$ belongs to $c_0X$ if there is an increasing sequence of finite sets $X_n$ such that $\cup X_n = X$ and $\|f\|_{\ell^\infty (X_n^\comp)} \to 0$. In pagan words, ``it decreases to $0$ at infinity''.

One of the spaces that will be used are the spaces of cuts. Roughly said, these are the image of the gradient. Intuitively, this is because the gradient of the characteristic function of $A$ is the cut $\del A$ (the edges between $A$ and its complement $A^\comp$). However, there are \emph{a priori} two such subspaces: the first are closures of $\nabla  \ell^0V$. So define, for $1 \leq p < \infty$,
\[
\mb{k}_\infty = \srl{ \nabla \ell^\infty V}^{\ell^\infty E} \quad , \quad  \mb{k}_c = \srl{ \nabla \ell^0V}^{c_0 E} \quad\text{and}\quad \mb{k}_p = \srl{ \nabla \ell^0V}^{\ell^p E}.
\]
Note that because $\nabla$ is bounded and $\ell^0V$ is dense in $\ell^pV$, $\srl{ \nabla \ell^0V}^{\ell^p E} = \srl{ \nabla \ell^pV}^{\ell^p E}$ (the same holds for $c_0$ but not $\ell^\infty$). On the other, one could also consider all the gradients which are in $c_0(E)$ or $\ell^pE$: define, for $1 \leq p \leq \infty$,
\[
\mb{K}_c = \nabla \rr^V \cap c_0 E  \text{ and } \mb{K}_p = \nabla \rr^V \cap \ell^p E
\]
The difference between these spaces is the reduced $c_0$ or $\ell^p$-cohomology in degree one. 

Next, comes the space of flows. There are again two possibilities. Let $C$ be the set of [finite oriented] cycles in the graph. As in homology, one can define on each cycle the boundary operator, here denoted by $\nabla_2^*$. Then, a first possibility for the space of flows is to consider\footnote{In this generality, it is not clear to the author what $\mb{f}_\infty$ should be}: for $1 \leq p < \infty$,
\[
\mb{f}_c = \srl{ \nabla_2^* \ell^0C}^{c_0 E} \text{ and } \mb{f}_p = \srl{ \nabla_2^* \ell^0C}^{\ell^pE}.
\]
The second possibility is, for $1 \leq p \leq \infty$.
\[
\mb{F}_c = \ker \nabla^* \subseteq c_0 E \text{ and } \mb{F}_p = \ker \nabla^* \subseteq \ell^p E.
\]
\vspace*{-2em}
\begin{rmk}\label{rnabla2}
Sometimes there is a good basis for $C$. Namely, there is a collection of cycles $C'$ so that every cycle can be decomposed as a sum of cycles in $C'$ and there is a uniform bound on the length of the cycles in $C'$. This property implies that a given edge belongs to finitely many element of $C'$. When looking at the Cayley graph of a group, this property coincides with the fact that the group is finitely presented. 

In those cases, $\nabla_2$ is a bounded operator from $\ell^p C'$ to $\ell^p E$. It is then relatively straightforward to check that $\mb{K}_p = \ker \nabla_2 \subseteq \ell^p E$ (and likewise for $c_0$). There is also a natural candidate for $\mb{f}_\infty$, namely $\srl{\nabla_2 \ell^\infty C'}^{\ell^\infty E}$.

On the other hand, $\nabla_2 \rr^{C'} \cap \ell^pE \neq \mb{F}_p$ (and likewise for $c_0$). The simplest case of a graph where this is an inequality happens when there are no cycles (\ie a tree, \eg the Cayley graphs of a free group with respect to a free generating set). So this might be seen as a third possibility for a space of cycles. See Question \ref{qnabla2}.
\end{rmk}

Recall that the space of harmonic functions with gradient in $c_0$ can be identified, modulo the constant functions, with $\mb{F}_c \cap \mb{K}_c$ (by looking at the gradient of these functions). Indeed, since the graph is connected, one can always ``integrate'' a gradient back to a function, up to a constant function. Hence, an element of $\mb{K}_c$ is (as a function) harmonic if and only if it lies also in $\mb{F}_c$. This ``proves'' (one can essentially take this as a definition):
\begin{lem}\label{tdefharm-l}
There are no non-constant harmonic functions with gradient in $c_0$
if and only if $\mb{F}_c \cap \mb{K}_c = \{0\}$. 
\end{lem}

Likewise, the space of Lipschitz harmonic function is $\mb{F}_\infty \cap \mb{K}_\infty$. In a group, it is never trivial by Shalom \& Tao \cite{ST}.

\subsection{Some reductions \emph{via} reduced cohomology in degree $1$}\label{sscoho}

The difference between $\mb{K}_1$ and $\mb{k}_1$ is called the reduced $\ell_1$-cohomology in degree $1$ (likewise for $c_0$). It is a well-known fact that the dimension of the quotient vector space $\mb{K}_1/\mb{k}_1$ is the number of ends of the graph minus $1$ (see Proposition A.2 in \cite{Go}). This statement is given a number here for further uses:
\begin{cor}\label{toneend-c}
Assume $G$ is infinite. Then $G$ has one end if and only if $\mb{K}_1 = \mb{k}_1$. 
\end{cor}

Equalities in the dual and pre-dual are easier to check:
\begin{prop}\label{tc0triv-p}
$\mb{k}_c = \mb{K}_c$.
\end{prop}
Before moving to the proof, let us first recall the key lemma (see \cite[Lemma 2.1]{Go}) for the convenience of the reader. This Lemma was inspired to the author by Lemma 4.4 from Holopainen \& Soardi \cite{HS} , which itself builds on a classical truncation lemma on Dirichlet harmonic functions.
\begin{lem}\label{tholoso-l}
Let $g \in \mb{K}_c$ be such that $g \notin \mb{k}_c$. Consider $g$ as a function $g:V \to \rr$. For $t \in \rr_{>0}$, let $g_t$ be defined as
\[
g_t(x) = \left\{ \begin{array}{ll}
g(x) & \text{if } |g(x)| < t, \\
t \tfrac{g(x)}{|g(x)|} & \text{if } |g(x)| \geq t.
\end{array} \right.
\]
Then there exists $t_0$ such that $g_t \notin \mb{k}_c$, for any $t > t_0$. In particular, $\mb{K}_c = \mb{k}_c$ if and only if $\mb{K}_c \cap \nabla\ell^\infty V = \mb{k}_c \cap \nabla\ell^\infty V$ 
\end{lem}
\begin{proof}
Although $g$ and $g_t$ are seen as functions, their norms is still given by their gradient. Assume, without loss of generality that $g(\mf{o}) =0$ for some preferred vertex (\ie root) $\mf{o} \in V$. Given $v \in V$ and $P$ a path from $\mf{o}$ to $v$, 
\[
|g(v)| = |g(v) - g(\mf{o})| = \sum_{e \in P:\mf{o} \to v} \nabla g(e) \leq d(\mf{o},v) \|\nabla g\|_{\ell^\infty E}. 
\]
In particular, $g_t$ is identical to $g$ on $B_{t/K}$ where $K = \| \nabla g \|_{\ell^\infty E}$ and $B_t$ is the ball of radius $t$ around $\mf{o}$ (in the combinatorial graph distance). Hence $\| \nabla g - \nabla g_t\|_{\ell^\infty E} \leq \|\nabla g\|_{\ell^\infty(B_{t/K}^\comp)}$, where $\ell^\infty(B_{t/K}^\comp)$ denotes the $\ell^\infty$-norm restricted to edges which are not inside $B_{t/K}$. Because $\nabla g \in c_0E$, $ \|\nabla g\|_{\ell^\infty(B_{t/K}^\comp)}$ tends to $0$, as $t$ tends to $\infty$.

Now if there is a infinite sequence $t_n$ such that  $g_{t_n}$ are in $\mb{k}_c$ and $t_n \to \infty$, then $g_{t_n}$ is a sequence of functions in $\mb{k}_c$ which tends (in $c_0 E$-norm) to $g$. This implies $g \in \mb{k}_c$, a contradiction. Hence, for some $t_0$, $g_t \notin \mb{k}_c$ given that $t >t_0$. 
\end{proof}

\begin{proof}[Proof of Proposition \ref{tc0triv-p}]
Given $f: V \to \rr$ with $\nabla f \in c_0E$, the aim is to show that $\nabla f \in \mb{k}_c$, \ie that the gradient of $f$ can be approximated by the gradient of finitely supported functions. By Lemma \ref{tholoso-l} (see also \cite[Lemma 2.1]{Go}), one can assume that $f \in \ell^\infty V$ without loss of generality.

So let $f \in \ell^\infty V$ with $\nabla f \in c_0 E$. Fix some ``root'' vertex $\mf{o}$ again and let $B_n$ be the ball of radius $n$ (for the combinatorial distance on the graph) around $\mf{o}$. Let $n_\eps$ be so that, for any $n> n_\eps$, $\|\nabla f\|_{\ell^\infty(B_n^\comp)} \leq \eps$ where $B_n^\comp$ is the complement of $B_n$. Let $g_\eps$ be identical to $f$ on $B_{n_\eps}$. Let $S_n := B_{n+1} \setminus B_n$. Then, using $\delta = \eps/ \|f \|_{\ell^\infty V}$, let 
\[
g_\eps(x) = \left\{ \begin{array}{ll}
f(x) & \text{if } x \in B_{n_\eps} \\
\big(1- (n-n_\eps) \delta \big) f(x) & \text{if } x \in S_n \text{ and } (n-n_\eps)\delta \in ]0,1[ \\
0 & \text{if } x \in S_n \text{ and } (n-n_\eps)\delta \geq 1 \\
\end{array}\right.
\]
Note that $g_\eps$ is finitely supported hence in $\ell^0V$. We claim that $\nabla g_\eps \overset{\eps \to 0}\longrightarrow \nabla f$. This follows by checking that $\| \nabla g_\eps -\nabla f\|_{\ell^\infty} \leq 3 \eps$. To do so, separate in cases depending on where the edge $(x,y)$ (or $(y,x)$, the arguments are symmetric) lies:\\
$\cdot$ $g_\eps-f \equiv 0$ inside $B_{n_\eps}$, so $\nabla g_\eps - \nabla f = \nabla (g_\eps -f) \equiv 0$ too; \\
$\cdot$ Let $n_0 = \lceil \delta^{-1}\rceil + n_\eps  -1$. On $B_{n_0}^\comp$, $g_\eps \equiv 0$ and $\|\nabla f\|_{\ell^\infty(B_{n_0}^\comp)} \leq \eps$ so $\| \nabla g_\eps - \nabla f\|_{\ell^\infty(B_{n_0}^\comp)} \leq \eps$;\\ 
$\cdot$ when $x,y \in S_n$ and $n \in [n_\eps, \delta^{-1} + n_\eps[$, $|\nabla f(x,y)| < \eps$ and $|\nabla g_\eps(x,y)|\leq |\nabla f(x,y)|$ (the scaling is the same inside a given sphere) hence $|\nabla g_\eps(x,y) - \nabla f(x,y)| \leq 2 \eps$;\\ 
$\cdot$ when $x \in S_n$, $y \in S_{n+1}$ and $n \in [n_\eps, \delta^{-1} + n_\eps[$. Then $\nabla f(x,y) < \eps$ while $|g_\eps(y) - g_\eps(x)| \leq  |\nabla f(x,y)|  + \delta |f(x)|$. Hence 
\[
|\nabla g_\eps(x,y) -\nabla f(x,y)| \leq  2 |\nabla f(x,y)| +  \delta |f(x)| < 3 \eps,
\]
by the choice of $\delta$.

To sum up, for any $\eps>0$, there is a finitely supported function $g_\eps$ so that $\| \nabla g -\nabla f\|_{\ell^\infty} \leq 3 \eps$. 
This implies $f \in \srl{ \nabla \ell^0V}^{c_0E} = \mb{k}_c$ and proves the claim.
\end{proof}
In case the reader is curious, the above equality is always false in $\ell^\infty$ (unless $G$ is finite): $\mb{K}_\infty \neq \mb{k}_\infty$. See \cite[Proposition A.3]{Go} for details.
However
\begin{prop}\label{tkinf-p}
$\srl{\mb{k}_\infty}^* = \mb{K}_\infty$
\end{prop}
Before moving on to the proof, let us quickly recall a basic fact about weak$^*$-convergence:
\begin{lem}\label{tweakstar-l}
Let $p \in [1,\infty]$ and $X$ be a countable set. The sequence $\{y_n\}_{n \in \nn}$ is weak$^*$ convergent in $\ell^pX$ if and only it is bounded in $\ell^pX$ and point-wise convergent.
\end{lem}
\begin{proof}[Proof of Proposition \ref{tkinf-p}]
It is easy to see from either Lemma \ref{tweakstar-l} or Lemma \ref{tdual-l} that $\mb{K}_\infty$ is weak$^*$-closed, hence $\srl{\mb{k}_\infty}^* \subseteq \mb{K}_\infty$.

For the other direction, the proof essentially goes as in the proof of Proposition \ref{tc0triv-p}. Note that by Lemma \ref{tweakstar-l}, it suffices to show that any $f \in \mb{K}_\infty$ is a point-wise limit of a bounded sequence $g_n$. To do so, it is (again) more convenient to think of these as functions (and the norm is taken on their gradient). Consider $f_t$ as in Lemma \ref{tholoso-l} by truncating $f$. As in said lemma, $f_t \equiv f$ on $B_{t/K}$ where $K= \|\nabla f\|_{\ell^\infty}$, so $\nabla f_t \to \nabla f$ point-wise. On the other hand, $\| \nabla f_t\|_{\ell^\infty} \leq \|\nabla f\|_{\ell^\infty}$ and $f_t \in \mb{k}_\infty$. Hence $f_t$ is a bounded point-wise convergent sequence of elements of $\mb{k}_\infty$. Thus, $f \in \srl{\mb{k}_\infty}^*$.
\end{proof}

\subsection{Annihilators and duals}

Some facts about annihilators will be required. The reader unfamiliar with these is invited to read \cite[{\S}4.8]{Rud}.
Our first step here is to identify the annihilators of cuts and flows.
\begin{lem}\label{tdual-l}
\[
\begin{array}{ccccccccccccc}
		&	& \mb{K}_1 &=& \mb{f}_c^\perp &\supseteq & \mb{F}_c^\perp &=& \srl{\mb{k}_1}^* &\supseteq & \mb{k}_1 &=& \mb{F}_\infty^\perpud \\
\mb{k}_\infty^\perpud &= & \mb{F}_1 &=& \mb{k}_c^\perp &\supseteq & \mb{K}_c^\perp &=& \srl{\mb{f}_1}^* &\supseteq & \mb{f}_1 &=& \mb{K}_\infty^\perpud
\end{array}
\]
\end{lem}
\begin{proof}
First, one shows the equalities $\mb{k}_c^\perp = \mb{F}_1$, $\mb{k}_1^\perp = \mb{F}_\infty$, $\mb{k}_1^\perpud = \mb{F}_c$ and $\mb{k}_\infty^\perpud = \mb{F}_1$.
These are all instances of the identity $(\img L)^\perp = \ker L^*$ (where $L$ is a bounded operator). For example,
\[
\begin{array}{rl}
\mb{k}_c^\perp 
&= (\nabla \ell^0V)^\perp\\
&= \{ y \in \ell^1E \mid \forall x \in \nabla \ell^0V, \pgen{y \mid x } = 0\}  \\
&= \{ y \in \ell^1E \mid \forall x' \in \ell^0V, \pgen{y \mid \nabla x' } = 0\}  \\
&= \{ y \in \ell^1E \mid \forall x' \in \ell^0V, \pgen{\nabla^* y \mid x' } = 0\}  \\
&= \{ y \in \ell^1E \mid \forall v \in V, \nabla^* y(v) = 0\}  \\
&= \ker \nabla^* \subset \ell^1E\\
&= \mb{F}_1.
\end{array}
\]
From those four equalities, one gets: $\mb{k}_c = \mb{F}_1^\perpud$, $\mb{k}_1 = \mb{F}_\infty^\perpud$, $\srl{\mb{k}_1}^*= \mb{F}_c^\perp$ and $\srl{\mb{k}_\infty}^* = \mb{F}_1^\perp$.

To show $\mb{f}_c^\perp = \mb{K}_1$, one needs to argue slightly differently\footnote{If $\nabla_2$ is available as a bounded operator, then the proof is absolutely identical.}. First, if $y \in \mb{K}_1$ then it sums to zero on any oriented cycle, hence any finite combination of such, hence a dense subspace of $\mb{f}_c$ and so, by continuity, $y$ is in the annihilator of $\mb{f}_c$: $\mb{K}_1 \subseteq \mb{f}_c^\perp$. Second, if $x \in \ell^1 E$ is in $\mb{f}_c^\perp$, then it sums to zero on any oriented cycle. This is exactly the condition that $x$ is the gradient of some function, hence $\mb{f}_c^\perp \subseteq \mb{K}_1$.

The equalities $\mb{f}_1^\perp = \mb{K}_\infty$ and $\mb{f}_1^\perpud = \mb{K}_c$ are obtained by the same arguments. By taking [weak-]annihilators, one gets:  $\mb{f}_c = \mb{K}_1^\perpud$, $\mb{f}_1 = \mb{K}_\infty^\perpud$ and $\srl{\mb{f}_1}^*= \mb{K}_c^\perp$.

Inclusions are presented in the statement so as to be obvious.
\end{proof}
Combining Lemma \ref{tdual-l} with Corollary \ref{toneend-c}, Proposition \ref{tc0triv-p} and Proposition \ref{tkinf-p}:
\[
\begin{array}{r@{\;}c@{\;}l@{\qquad}r@{}c@{}l@{\;}c@{\;}l@{\qquad}r@{\;}c@{\;}l@{}c@{}l}
\mb{k}_c & = & \mb{K}_c, & \mb{k}_1 &\overset{o.e.?}\subsetneq 	& \mb{K}_1,	  &   & 	& \mb{k}_\infty &\subsetneq &\srl{\mb{k}_\infty}^* &=				& \mb{K}_\infty\\
         &   &           & \mb{f}_1 &= 				&\srl{\mb{f}_1}^* & = & \mb{F}_1&               &	    &\srl{\mb{f}_\infty}^* &\underset{o.e.?}\subsetneq 	& \mb{F}_\infty\\
\end{array}
\]
where 
$\mb{f}_\infty$ is only defined if $\nabla_2^*$ can be defined as a bounded operator,
and $\overset{o.e.?}\subsetneq$ means the inclusion is strict if and only if the graph has $>1$ ends.

Note that there are probably more direct proofs of some of the equalities above. For example, it is easy to check that $\img \nabla_2^* \cap \ell^1E = \mb{F}_1$. From there, it is not too difficult to conclude $\mb{f}_1 = \mb{F}_1$.

The ``dual'' criterion for the absence of non-constant harmonic functions with gradient in $c_0$ can now be proven. In order to alleviate notations, the shorthand ``has $\hdc$'' will be used in order to say ``has no non-constant harmonic functions with gradient in $c_0$''.

\begin{lem}\label{tanni-l}
$G$ has $\hdc$ if and only if $\ell^1E = \srl{\mb{k}_1 + \mb{F}_1}^*$. \\
If $G$ has one end, then $G$ has $\hdc$ if and only if $\ell^1E = \srl{\mb{K}_1 + \mb{F}_1}^*$.
\end{lem}
\begin{proof}
Lemma \ref{tdefharm-l} shows that $G$ has $\hdc$ exactly when $\mb{F}_c \cap \mb{K}_c = \{0\}$.
Using the basic properties of annihilators $(A\cap B)^\perp = \srl{A^\perp + B^\perp}^*$ and that $\ell^1 E$ is the annihilator of $\{0\} \subset c_0E$, one gets
\[
\hdc \text{ holds} \iff \ell^1E = \srl{\mb{F}_c^\perp + \mb{K}_c^\perp}^*
\]
By proposition \ref{tc0triv-p}, $\mb{K}_c = \mb{k}_c$. Lemma \ref{tdual-l} then implies $\mb{K}_c^\perp = \mb{k}_c^\perp = \mb{F}_1$ and $\mb{F}_c^\perp = \srl{\mb{k}_1}^*$. Hence
\[
\hdc \text{ holds} \iff \ell^1E = \srl{\mb{k}_1 + \mb{F}_1}^*
\]
By Corollary \ref{toneend-c} (see also directly \cite[Proposition A.2]{Go}), if $G$ has one-end then $\mb{K}_1 = \mb{k}_1$, so 
\[
\text{one-end implies } \qquad \hdc \text{ holds} \iff \ell^1E = \srl{\mb{K}_1 + \mb{F}_1}^* \qedhere
\]
\end{proof}
\begin{rmk}
~\\
$\cdot$ Since all infinite groups have a non-trivial Lipschitz harmonic function (see Shalom \& Tao \cite{ST}), $\mb{K}_\infty \cap \mb{F}_\infty \neq \{0\}$ and by the same arguments as in the proof of Lemma \ref{tanni-l}, $\srl{\mb{k}_1 + \mb{f}_1} = \srl{\mb{k}_1 + \mb{F}_1} \subsetneq \ell^1E$. If further $G$ is one-ended, then $\srl{\mb{K}_1 + \mb{F}_1} \subsetneq \ell^1E$. \\
$\cdot$ Since there are one-ended groups (\eg $\zz^2$) with $\hdc$, there are groups where $\srl{\mb{K}_1 + \mb{F}_1} \subsetneq \srl{\mb{K}_1 + \mb{F}_1}^*  = \ell^1E$.  \\
$\cdot$ Since any two-ended group (\eg $\zz$) has $\hdc$, there are groups where $\srl{\mb{k}_1 + \mb{F}_1} \subsetneq \srl{\mb{k}_1 + \mb{F}_1}^* = \ell^1E$. In fact, in the usual Cayley graph of $\zz$ (the line), it is easy to see that $\mb{K}_1 = \ell^1E$ while $\mb{F}_1 = \emptyset$. Hence there are groups where $\srl{\mb{k}_1 + \mb{F}_1} \subsetneq \srl{\mb{k}_1 + \mb{F}_1}^* = \srl{\mb{K}_1 + \mb{F}_1} = \ell^1E$. \\
$\cdot$ In a the Cayley graph of group with infinitely many ends (\eg a regular tree), one can check that $\srl{\mb{k}_1}^* \subsetneq \mb{K}_1$. To do so, consider a half-tree $H$ (a connected component after removing an edge) and look at $\nabla \un_H \in \mb{K}_1$. The fact the half-tree has a strong isoperimetric ratio function ($\mathrm{IS}_\omega$, see \S{}\ref{ssisop}) implies that $\nabla \un_H \notin \srl{\mb{k}_1}^*$.
\end{rmk}

The proof of Lemma \ref{tdual-l} can be mimicked without problem in the reflexive case. The output is cleaner, since closure and weak$^*$-closure coincide.
\begin{lem}\label{tduallp-l}
Let $p' = \tfrac{p}{p-1}$ be the H{\"o}lder conjugate of $p \in ]1,\infty[$, then
\[
\begin{array}{ccccccc}
\mb{K}_p &=& \mb{f}_{p'}^\perp &\supseteq & \mb{F}_{p'}^\perp &=& \mb{k}_p\\
\mb{F}_p &=& \mb{k}_{p'}^\perp &\supseteq & \mb{K}_{p'}^\perp &=& \mb{f}_p\\
\end{array}
\]
\end{lem}
Lemma \ref{tanni-l} also has an analogue for $\ell^p$ (proved using Lemma \ref{tduallp-l}). Here ``has $\hdp$'' means ``has no non-constant harmonic functions with gradient in $\ell^p$''.
\begin{lem}\label{tannilp-l}
$G$ has $\hdp$ if and only if $\ell^{p'}E = \srl{\mb{k}_{p'} + \mb{f}_{p'}}^*$.
\end{lem}

\begin{rmk}
If a group has a presentation with generators $s_1, \ldots s_n$ none of which are of order 2 and relations $t_1^{n_1} t_2^{n_2}\cdots t_k^{n_k}$ so that $t_i \in S$ and (for each relator) $\sum n_i = 0$, then $\mb{k}_1$ and $\mb{f}_1$ both lie in a weak$^*$ closed subspace of $\ell^1E$. 
Indeed, orient the edges so that $s_i$ is the positive direction (so $s_i^{-1}$ is the negative direction). Then $\nabla \delta_x$ has zero sum. Also, for any cycle $c$, $\nabla_2^* c$ has exactly the same number of positive and negative edges. Hence $\nabla_2^* c$ also has zero sum. 
This implies that $\mb{k}_1 + \mb{f}_1 \subset \ell^1_0E \subsetneq \ell^1E$.

Note that it is obvious that such groups have non-constant Lipschitz harmonic functions.
Indeed, they have an infinite Abelianisation: a surjection $\Gm \surj \zz$ is given by sending each generator to $1 \in \zz$.
\end{rmk}

\subsection{Liouville property}\label{ssliou}

A graph $G$ is called \textbf{Liouville} if there are no non-constant bounded harmonic functions, \ie $(\nabla \ell^\infty V) \cap \mb{F}_\infty = \{0\}$.
\begin{rmk}\label{rliou}
It is not clear that there is an equivalence between  ``$G$ is Liouville'' and ``$\mb{F}_\infty \cap \mb{k}_\infty = \{0\}$''. Indeed, as was shown in Proposition \ref{tc0triv-p}, $\mb{k}_\infty$ contains much more than $\nabla \ell^\infty V$.
See Questions \ref{qisom} and \ref{qbdd}.
\end{rmk}

The following characterisation (due to V. Kaimanovich) will be useful for our upcoming purposes.
In order to shorten the notation, let $\mb{D}_1 = \srl{\Delta \ell^1V}^{\ell^1V}$ be the closure of the image of the Laplacian.
Also, recall that a graph is called Liouville if it has no non-constant harmonic functions.

The formulation which will be interesting here is a characterisation of the Liouville property in terms of exit distributions.
For $A \subset V$, denote by $\delta A$ the set of vertices in $A^\comp$ which are neighbours to some vertex of $A$. 
Given a finite connected subset $A$ and a vertex $v \in A$, define the exit probability $\mr{ex}_v^A: \delta A \to [0,1]$ by $\mr{ex}_v^A(x)$ is the probability that a simple random walk starting at $v$ lands in $x$ the first time it exits the set $A$.

It turns out there is another way of obtaining this probability measure.
Take $P_A$ to be the random walk operator ``stopping outside $A$'' (it's a Markovian operator). 
It is defined on the Dirac mass basis of $\ell^1V$ as follows: $P_A \delta_x = \tfrac{1}{d} \un_{ \{ y \mid y \sim x\} }$ if $x \in A$  and $P_A \delta_x = \delta_x$ if $x \notin A$. 
In other words, the random walker only walks as long as it is in $A$ and stays where he is once he has left $A$.

The author learned about the following result on Math Overflow, thanks to V.~Beffara (see \cite{Beff-MO}). This theorem is a particular case of a result of V.~Kaimanovich (see \cite[Theorem 2.6]{Kai})
\begin{teo}\label{tliouex-t}
$G$ is Liouville if and only if the following holds. For any increasing sequence of finite connected sets $A_n$ so that $\cup A_n = V$ and any vertices $v,w \in V$, 
\[
\lim_{n \to \infty} \| \mr{ex}_v^{A_n} - \mr{ex}_w^{A_n} \|_{\ell^1} =0.
\]
\end{teo}
Note that the first few terms of the sequence might not be defined, but there is a $N$ so that, for $n>N$, they all are defined. 
Also one could replace the $\lim$ by a $\liminf$ without changing anything.

\subsection{Isoperimetry}\label{ssisop}

\newcommand{\IS}[0]{\mathrm{IS}}

\newcommand{\Gc}[0]{\widetilde{\jo{G}}}
\newcommand{\Go}[0]{\jo{G}_0}
\newcommand{\Gd}[0]{\jo{G}^{\scalebox{.4}{$\searrow$}}}

This section is a short review of isoperometry on groups and graphs. Woess' book offers a much more complete reference \cite{Woe}; see also a survey by Pittet \& Saloff-Coste \cite{PSC-survey} or Coulhon \& Saloff-Coste paper \cite{CSC}.

The \textbf{isoperimetric ratio function} is the function $\jo{F}: \zz_{>0} \to \zz_{\geq 0}$ defined to be the largest function so that $\jo{F}(|F|) \leq |\del F|$, \ie $\jo{F}(x) = \inf \{ |\del F| \, \mid \, |F|= x\}$.
The \textbf{isoperimetric ratio function} is the function $\jo{G}: \zz_{> 0} \to \rr_{\geq 0}$ defined by $\jo{G}(x) := x^{-1}\jo{F}(x) = \inf \{ \frac{|\del F|}{|F|} \, \mid \, |F|= x\}$.
Be aware that some texts use ``isoperimetric ratio function'' in place of ``isoperimetric ratio function''
The (decreasing) function $\Gd(x) := \inf \{ \frac{|\del F|}{|F|} \, \mid \, |F|\leq  x\}$ derived from $\jo{G}$ will also come in handy.

Looking at the decreasing function $\Gd$ instead of $\jo{G}$ is not much of thing.
Recall that $\jo{F}$ is subadditive: $\jo{F}(a+b) \leq \jo{F}(a) + \jo{F}(b)$. Hence, up to a minor change (taking the concave hull of a subadditive function changes it by at most a factor $2$), $\jo{F}$ is concave.
If $f$ is concave, then $x^{-1} f(x)$ is decreasing.

A graph has a ... \\
\begin{tabular}{@{\textbullet\;}llll}
strong isoperimetric ratio& ($\IS_\omega$) & if $\exists K >0$ so that & $\jo{G}(x) \geq K $. \\
$d$-dimensional isoperimetric ratio&($\IS_d$) &if $\exists K >0$ so that &$\jo{G}(x) \geq \frac{K}{x^{1/d}}$. \\
\end{tabular}

A strong isoperimetric ratio is equivalent to the statement that $\nabla$ has a closed image. For $\ell^1$ this is direct from the definition, and it is classical that it is equivalent for any $p \in [1;+\infty[$.  $\nabla$ does not have closed image in $c_0$ (and consequently in $\ell^\infty$).

For groups amenability is the negation of $\IS_\omega$. However any group which is not of polynomial growth but is amenable will satisfy $\IS_d$ for any $d$ (see for example \cite{CSC}).

\section{Applications to $\hdp$} \label{shdp}

\subsection{Liouville graphs and direct products} \label{sslioudir}

Recall that Liouville is ``$\mb{F}_\infty \cap \nabla \ell^\infty V = \{0\}$'' and $\hdp$ is, by definition (or an analogue of Lemma \ref{tdefharm-l}), equivalent to ``$\mb{K}_p \cap \mb{F}_p =\{0\}$'' (i.e. no non-constant harmonic functions with gradient in $\ell^p$).
As $\mb{k}_p \subsetneq \nabla \ell^\infty V$, it is not trivial that Liouville implies $\hdp$ for all $p \in [1,\infty[$.
Theorem \ref{thdp-t} below was already proven in \cite[Theorem 1.2 or Corollary 3.14]{Go}, but the following proof is more direct (and avoids $\ell^p$-cohomology; see also \cite{Go-frontier} for a streamlined proof using $\ell^p$-cohomology).
An ingredient which will come in the proof is that of a \textbf{transport pattern} from $\xi$ to $\phi$ (where $\xi$ and $\phi$ are measures). 
This is a finitely supported function $\tau:E \to \rr$ so that $\nabla^* \tau = \phi - \xi$. 

Note that for any $h \in \ell^1E$, $\nabla^* h =: \pi$ can be decomposed into two positive measures $\pi_\pm \in \ell^1V$ so that $\pi = \pi_+ - \pi_-$ and $\|\pi_+\|_{\ell^1} = \|\pi_-\|_{\ell^1}$.
This follows from the fact that the image of $\nabla^*$ is contained in the functions summing to 0.
For any such $\pi = \pi_+ - \pi_-$ of finite support, 
let $TP(\pi)$ be the set of all (finitely supported) transport patterns with $\nabla^* \tau = \pi$.
\begin{lem}
Let $h$ be a finitely supported function on the edges. The norm $h$ in the quotient $\ell^pE / \mb{f}_p$ is $\inff{\tau \in TP(\nabla^*h)} \|\tau\|_{\ell^p}$ 
\end{lem}
\begin{proof}
By definition, the closure (in $\ell^pE$) of the set $TP(\nabla^* h)$ equals $h + \mb{f}_p$. The infimal norm of such an element is the quotient norm.
\end{proof}

In the particular case $p=1$, the quotient norm $\ell^1E / \mb{F}_1$ is the Wasserstein distance between $p_+$ and $p_-$.
Transport patterns will be a combination of [the characteristic function of oriented] geodesic paths between the support of $p_+$ and the support of $p_-$.
Still for $p=1$, the optimal transport patterns are automatically finitely supported if $p$ is. r

For $p>1$, the quotient norm $\ell^pE / \mb{f}_p$ has nothing to do with the $p$-Wasserstein distance. 
For example, when $p=2$, it expresses the energy of a current flowing between sinks and sources. 
The infimum is probably never realised by a finitely supported transport pattern.

\begin{rmk}\label{rtra}
It is quite important to point out that one can essentially always assume that $\|\tau\|_{\ell^\infty} \leq \|\pi_\pm\|_{\ell^1} = \tfrac{1}{2} \|\pi\|_{\ell^1}$.
Indeed, if any transport pattern $\tau$ (from $\pi_-$ to $\pi_+$) is such that $\|\tau\|_{\ell^\infty} > \|\pi_-\|_{\ell^1}$ then it means some mass is transported twice along some edge.
This implies that the said mass is transported along a cycle (something which is definitively not very efficient).
Hence there is another transport pattern $\tau'$ where this [useless from the view of transport] cycle is avoided and $ |\tau'| \leq |\tau|$ (point-wise). 
\end{rmk}

\begin{lem}\label{tlemwkst}
Let $G$ be $D$-regular graph. If for any neighbours $v$ and $w$ there is an increasing and exhausting sequence of sets $A_n$ and a sequence of transport patterns $\tau_n$ from $\mr{ex}_v^{A_n}$ to $\mr{ex}_w^{A_n}$ so that $\tau_n$ tends weak$^*$ to 0 in $\ell^{p'}E$, then $G$ has $\hdp$.
\end{lem}
\begin{proof}
The aim is to show that $\srl{\mb{k}_{p'}+\mb{f}_{p'}}^*$ contains all the Dirac masses $\delta_{v,w}$ of $\ell^{p'}E$. By density of the Dirac masses, it is then equal to $\ell^{p'}E$. 

To do so consider $f_n = \sum_{i \geq 0} P_{A_n}^i \nabla^* \delta_{v,w}$. 
$f_n$ is a function with bounded support, so $\nabla f_n \in \mb{k}_{p'}$.
Note that $\nabla^* (\nabla f_n - \delta_{v,w}) = \mr{ex}_w^{A_n} - \mr{ex}_v^{A_n}$.
By the property of the transport patterns $\tau_n$, $\nabla f_n - \delta_{v,w}$ belongs to the weak$^*$ closure of $\mb{f}_{p'}$.
\end{proof}

\begin{lem}\label{tlemtra}
Let $G$ be $D$-regular graph with $\IS_d$ for some $d>2$ and let $p <d/2$. 
Then for any neighbours $v$ and $w$ and any increasing and exhausting sequence of sets $A_n$ there is a sequence of transport patterns $\tau_n$ from $\mr{ex}_v^{A_n}$ to $\mr{ex}_w^{A_n}$ so that $\|\tau_n\|_{\ell^{p'}E}$ is bounded.
\end{lem}
\begin{proof}
The transport pattern $\tau_n$ presented here are extremely inefficient.
Construct $\tau_n$ as a concatenation of many transport patterns.
The first series of transport pattern bring the mass from $\mr{ex}_v^{A_n}$ to $\delta_v$. 
Note that the transport between $P^i_{A_n} \delta_y$ and $P^{i+1}_{A_n}\delta_y$ can be achieved by just making a random step (for the mass lying within $A_n$, the rest does not move): at every given vertex, split and move the mass evenly to all the neighbours. The $\ell^p$-norm of this ``random step'' transport is $\leq \| P^n_{A_n} \delta_y\|_{\ell^{p'}}$

So transport from $\delta_v$ to $\mr{ex}_v^{A_n}$ by this method.
This transport moves less mass than the transport from $P^i \delta_x$ to $P^{i+1} \delta_x$ (because one only moves what lies inside $A$).
Hence, the $\ell^{p'}$-norm (for the transport from $\delta_v$ to $\mr{ex}_v^{A_n}$) is bounded by $\sum_{i \geq 0} \| P^i \delta_y\|_{\ell^{p'}}$.

Using $\IS_d$, one has that $\|P^i \delta_y \|_{\ell^\infty} \leq K n^{-d/2}$ (for some $K>0$). 
A simple interpolation 
(one also has $\|P^i \delta_y\|_{\ell^1}=1$, and for any $f \in \ell^1\nn$, $\| f \|_q^q \leq \|f\|_1 \|f\|_\infty^{q-1}$) 
gives an upper bound on the $\ell^{p'}$-norm. 
So the $\ell^{p'}$-norm of this [very inefficient] transport pattern from $\mr{ex}_v^{A_n}$ to $\delta_v$ is bounded by $\sum_{i=0}^{\infty} \| P^n \delta_y \|_{\ell^{p'}}$.
It turns out that the series converges when $d > 2p$. 
The value of the series only depends on the constant in the profile $\IS_d$.

A similar transport pattern bring the mass from $\mr{ex}_w^{A_n}$ to $\delta_w$.
A last (and very obvious) transport pattern can bring the mass from $\delta_v$ to $\delta_w$.
By combining these three transport patterns, one gets $\tau_n$ and the $\ell^{p'}$-norm of $\tau_n$ is bounded (independently of $n$, $A_n$, $v$ and $w$).
\end{proof}

\begin{teo}\label{thdp-t}
Assume $G$ is $D$-regular Liouville graph with $\IS_d$ for some $d>2$ and let $p <d/2$.
Then $G$ has $\hdp$.
\end{teo}
\begin{proof}
By the Lemma \ref{tlemtra}, there is a plan $\tau_n$ which is bounded. 
However, this plan might not necessarily tend weak$^*$ to 0.
By Remark \ref{rtra}, one can check that there is a alternative plan $\tau'_n$ which is point-wise smaller than $\tau_n$ and so that $\| \tau'_n \|_{\ell^{\infty}} \leq \|\mr{ex}_v^{A_n}-\mr{ex}_w^{A_n}\|_{\ell_\infty}$. 
The latter implies $\tau'_n$ is bounded in norm 
(since $\| \tau'_n \|_{\ell^{p'}} \leq \| \tau_n \|_{\ell^{p'}}$
while the former together with a simpler form of Kaimanovich's 0-2 law (Theorem \ref{tliouex-t}) implies it tends to $0$ point-wise.
Boundedness together with point-wise convergence to $0$ is equivalent to weak$^*$ convergence (see Lemma \ref{tweakstar-l}).
The conclusion follows by Lemma \ref{tlemwkst}.
\end{proof}
In the proof of Theorem \ref{thdp-t}, it is unfortunately not possible to prove that $\tau'_n \to 0$ in the weak$^*$ topology on $\ell^1E$ (if so, it would imply the absence of non-constant harmonic functions with gradient in $c_0$).

Note that the previous result may also be found in \cite[Theorem 1.2]{Go} or \cite[Corollary 4.15]{Go-frontier}. By \cite[Theorem 5.12]{Go-mixing}, if this result can be applied to (the Cayley graph of) a subgroup which is not almost-malnormal, then it extends to the whole group.

Recall that if $G_1 = (X_1,E_2)$ and $G_2= (X_2,E_2)$ are graphs, their direct product is the graph $G$ with vertex set $X = X_1 \times X_2$ and the pair $(x_1,x_2)$ is a neighbour of $(y_1,y_2)$ if either $x_1 \sim x_2$ or $y_1 \sim y_2$ but not both.
It would be possible to use the techniques of this section to prove that direct products of graphs have $\hdp$ (a variant of  \cite[Proposition 2]{Go-cras}). However work of Amir, Gerasimova and Kozma
\cite{AGK} and Corollary \ref{tdirc0gp-c} contain stronger results. 

\begin{rmk}
As soon as a graph satisfies $\mr{IS}_d$ for $d > 2p'$, one has that $\srl{\mb{k}_p+ \mb{F}_p} = \ell^pE$. Indeed, under this hypothesis $\mf{G} \delta_x \in \ell^pV$ (here $\mf{G} = \sum_{i=0}^\infty P^i$ is Green's kernel) so that taking $H = \mf{G}(\delta_v - \delta_w)$ (for any neighbours $v \sim w$), one has
\[
\nabla H \in \mb{k}_p, \quad \nabla H - \delta_{v,w} \in \ker \nabla^* = \mb{F}_p.
\]
so $\delta_{v,w} \in \mb{k}_p + \mb{F}_p$. Taking closure one gets $\ell^pE$. See \cite[Lemma 4.13]{Go-frontier} for details.
\end{rmk}

\subsection{$\hdc$ and commuting elements}\label{ss-c0centr}

The main method of this section is to interpret the norm and the equivalence classes of elements from $\ell^p E/ \srl{\mb{k}_p + \mb{f}_p}$ and deduce the non-existence of harmonic functions with gradient conditions.

Recall that ``has $\hdc$'' is a shorthand for ``has no non-constant harmonic functions with gradient in $c_0$'' (like $\hdp$ but with $c_0$ in place of $\ell^p$).

Given a function $f \in \ell^p E$, recall that the norm of $f \in \ell^p E/\mb{f}_p$ is the norm of a transport pattern of $\nabla^* f$.
So it remains to interpret the contribution of $\mb{k}_p$. 
To this end consider the Dirac mass $\delta_v \in \ell^pV$. 
As an element of $\mb{k}_p$ it is sent to $\nabla \delta_v$.  
Hence its effect on the above norm is to change $\nabla^* f$ to $\nabla^* f + \nabla^* \nabla \delta_v$.

Seeing the function $\nabla^* f$ as pile of chips on each vertex,
this can be interpreted as ``firing up'' a vertex: [part of] the value of $\nabla^* f$ can be removed from a vertex and redistributed (equally) on all the neighbours. 
Note that this is possible even $\nabla^* f$ has value 0 at a vertex.
The random walk can be seen as a special case where one ``fires up'' (the full value at) all the vertices.

When looking at $\ell^1 E/ \srl{\mb{k}_1 + \mb{f}_1}^*$, note that it further suffices to show that there are bounded transport patterns that tend point-wise to 0 (see Lemma \ref{tweakstar-l} or the proof of Theorem \ref{thdp-t}).

\begin{lem}\label{texco-l}
Assume that for an exhausting and increasing sequence $A_n$ and for some  vertices $x$ and $y$, the $\ell^1$-transport distance of the $n^{\text{th}}$-step random walk distribution $P^n \delta_x$ and $P^n \delta_y$
is bounded independently of $n$.
Then any path from $x$ to $y$ belongs to $\srl{\mb{k}_1 + \mb{f}_1}^*$.
\end{lem}
\begin{proof}
Let $p_{xy} \in \ell^1E$ be such a path, then $\nabla^* p_{xy} = 
\delta_x - \delta_y$. 
Modulo $\mb{k}_1$ this belongs to the same class as  
$P^n \delta_x - P^n \delta_y$.
The norm of this class in  $\ell^1E / \mb{f}_1$ is bounded (by hypothesis).
Furthermore, since the $P^n$ tend point-wise to 0,
the (bounded) transport pattern must tend point-wise to 0.

Indeed, assume there some edge $e$ for which $\tau(e) > c$ independently of $n$. 
Let $k_n$ be the largest integer, such that a ball of radius $k_n$ around $e$ contains a mass of at most $c/2$ i.e. $P^n\delta_x ( B_{k_n}(e) ) \leq c/2$. 
Since $P^n\delta_x$ tends point-wise to 0, $k_n \to \infty$.
Then in order to transport the [at least] remaining $c/2$ mass, the transport pattern must transport this mass over at least a distance of $k_n$. Since $k_n c/2$ is not bounded, this contradicts the hypothesis.

This means that $p_{xy} \in \ell^1E / \srl{ \mb{k}_1 + \mb{f}_1}^*$.
\end{proof}
In particular, the previous lemma implies that any harmonic function with $c_0$-gradient will take the same values at the vertices $x$ and $y$. It is also possible to prove the previous lemma with the exit distributions (w.r.t. to some sequence of sets $A_n$), but the author could not apply the upcoming results to this distribution.

Although the required property is hard to ensure in a generic graph, note that if there is a central element in a group, then one gets many such pairs.

\begin{lem}\label{tcenco-l}
Let $\Gamma$ be a group and assume $z \in Z(\Gamma)$. Then in any Cayley graph of $\Gamma$ and for any $x \in \Gamma$,
the distributions $P^n\delta_x$ and $P^n\delta_{zx}$ stay at bounded transport distance.
\end{lem}
\begin{proof}
Since $z$ belongs to the centraliser, applying a step of the random walk, and then moving along the path [in the Cayley graph] labelled by writing $z$ as a word in the generators is the same as first moving along $z$ and the applying the random walk: since $P = \sum_s \delta_s$ and $\delta_s*\delta_z = \delta_z*\delta_s$, then $P*\delta_z = \delta_z * P$.

This means that the transport pattern from $P^n \delta_x$ to $P^n\delta_{zx}$ simply consists in following the path labelled by $z$. This plan is bounded by the word length of $z$.
\end{proof}

\begin{teo}\label{tinfcen-t}
Assume $\Gamma$ has an infinite centre. 
Then (in any Cayley graph) $\hdc$ holds.
\end{teo}
\begin{proof}
The aim is to use Lemma \ref{tanni-l}, namely to show that $\ell^1E = \srl{\mb{k}_1+\mb{f}_1}^*$.
Let $p_{xy}$ be a path between $x$ and $y$. 
Take $c_n$ to be a sequence of central elements (which tend to infinity, i.e. leaves any fixed ball).
Let $x_n = xc_n = c_nx$ and $y_n = yc_n = c_ny$.
Since $p_{xy} = p_{xx_n} + p_{x_ny_n} + p_{y_ny}$.
By Lemmas \ref{texco-l} and \ref{tcenco-l}, this means that $p_{xy}$ and $p_{x_ny_n}$ belong to the same equivalence class.
Note that $\|p_{x_ny_n}\|_{\ell_1E} = \|p_{xy}\|_{\ell^1E}$, since the distance from $x_n$ to $y_n$ is the same as the distance from $x$ to $y$.
Lastly, since $p_{x_ny_n}$ is bounded and tends point-wise to 0, one get that $p_{xy}$ is in the same class as 0 in $\ell^1E / \srl{\mb{k}_1+\mb{f}_1}^*$. 
\end{proof}

Note that this result can be improved to include the case where there are infinitely many elements with a finite conjugacy class. To do this one just needs to adapt the proofs of \cite[Proposition 1.5 and Lemma 2.7]{GJ}.

The reader is directed to the work of Amir, Gerasimova and Kozma \cite{AGK} for a proof that direct product of graphs have $\hdc$.

In the case of two groups $G_1$ and $G_2$, note that the direct product $G= G_1 \times G_2$ of the groups admits many generating sets for which the resulting graph is not a direct product of the Cayley graphs (in the sense of graph products). See Corollary \ref{tdirc0gp-c} for a result about direct product of groups.

\subsection{Centralisers and malnormal subgroups}\label{sscentmal}

This section deals only with Cayley graph of finitely generated groups.
Although the main focus is for harmonic function in the usual sense, it will be useful to consider functions which are harmonic with respect to other measures. It will be assumed that the measure is finitely supported and symmetric (\ie $\mu(s) = \mu(s^{-1})$); see Remark \ref{rfinsup} for some variations.

Recall that the centraliser of $H$ in $G$ is 
$Z_G(H) := \{ g \in G \mid \forall h \in H, gh = hg\}$.
If $[g,H]$ denotes the set $\{ [g,h] \mid h \in H\}$, note that this can be rewritten as 
$Z_G(H) := \{ g \in G \mid  [g,H]$ has one element $\}$.
This second definition is closer to the definition of the FC-centraliser of $H$ in $G$: 
$Z_G^{FC}(H) := \{ g \in G \mid [g,H]$ is finite $\}$.
\begin{lem}\label{tconcen-l}
Let $H<G$ be an infinite subgroup of $G$. 
Let $f$ be a function on $G$ with gradient in $c_0$. 
Assume $f$ is harmonic with respect to the [finitely supported symmetric] measure $\mu$ whose support generate an infinite subgroup of $H$.
Then $f$ is constant on the right cosets of $Z= Z_G(H)$ or $Z^{FC}_G(H)$ (\ie $f$ is constant on the sets $g Z$ for any $g \in G$).
\end{lem}
\begin{proof}
For readability, the case $Z = Z_G(H)$ will be done first.
Consider $g \in G$ and $z \in Z_G(H)$. 
Then let $d = h(g) - h(gz) = \langle h | \delta_g - \delta_{gz} \rangle$.
Since $h$ is harmonic with respect to $\mu_H$, $d= \langle h | \delta_g * \mu_H^n - \delta_{gz} * \mu_H^n \rangle$.
However, since $z$ lies in the centraliser, $\delta_{gz} *\mu_H^n = \delta_g *\delta_z * \mu_H^n = \delta_g * \mu_H^n* \delta_z$.
This means there is a transport pattern from $\delta_g * \mu_H^n$ to $\delta_{gz} * \mu_H^n$ which is bounded (by the word length of $z$) and tends to 0 (since $H$ is infinite and the support of $\mu_H$ generates $H$).
Consequently, $d =0$ which means that $h$ is constant on the right cosets of $Z_G(H)$. 

To get the conclusion for $Z= Z_G^{FC}(H)$, note that 
$\delta_z * \mu_H 
= \delta_z * \sum_{h \in H} \mu_H(h) \delta_h 
= \sum_{h \in H} \mu_H(h) \delta_z * \delta_h 
= \sum_{h \in H} \mu_H(h) \delta_h * \delta_z * \delta_{c(h)}$
where $c(h) = [z^{-1},h^{-1}]$ depends on $h$, but takes only finitely many values.
So the transport pattern is still bounded (by $\max_{h \in H} |z c(h)|_S$) and the conclusion also holds.
\end{proof}

There are some further definitions and properties which should be introduced before moving on to the next lemma.
A subgroup $K<G$ is almost-malnormal if $\forall g \in G \setminus K$, $K \cap g K g^{-1}$ is finite. (The whole group $G$ is an almost-malnormal subgroup of itself.)
Recall that the q-normaliser of $H$ in $G$ is $N^q_G(H) = \langle g \in G \mid H \cap g H g^{-1}$ is infinite $\rangle$ (this subgroup is not always finitely generated).

Note that, if $Z = Z_G(H)$ or $Z_G^{FC}(H)$, then $H \subset N^q(Z)$ and $Z \subset N^q(H)$.
Also if one starts at a subgroup $K$ and considers the [transfinite] sequence of iterated q-normalisers of $K$, then this sequence stabilises at the almost-malnormal hull of $K$ (the smallest almost-malnormal subgroup containing $K$).
Furthermore, the almost-malnormal hull of a finite subgroup is itself.
Lastly, if $L < K$ is an infinite subgroup and $L \lhd H < G$ then $N^q(K)$ (and hence the almost-malnormal hull) contains $H$.

\begin{lem}\label{tconmar-l}
Let $G$ be an infinite groups, $S$ a finite generating set and consider the associated Cayley graph.
Assume $f$ is so that $\nabla f \in c_0E$ and $f$ is constant on the right-cosets of some subgroup $K<G$.
Then $f$ is constant on the almost-malnormal hull of $K$.
\end{lem}
\begin{proof}
Let $N = N^q_G(K)$ and take $g \in N$ to be one of the generators of $N$, \ie $K \cap gKg^{-1}$ is infinite. 
Then $gK \cap Kg$ is infinite, so let $g_n$ be a infinite sequence of (distinct) elements in this intersection.
Note that $\gamma_n \in gK$, hence $d = f(\gamma_n)$ is a constant.
On the other hand all the $\gamma_n$ are at distance at most $|g|_S$ from $K$, because $\gamma_n \in Kg$. 
Since $f$ is constant on $K$, it follows that $f(\gamma_n) = f(K) + \sum_{k \in \pi_{K,g_n} } \nabla f(k)$ where $\pi_{K,\gamma_n}$ is some path of length at most $|g|_S$ from $K$ to $\gamma_n$.
But $f \in c_0E$ and consequently, the sum tends to 0 as $n \to \infty$. 
Hence $f(K) = f(\gamma_n) = f(gK)$.

By induction this can be extended to any element of $g$. Indeed, assume that $f$ is constant on all cosets of the form $g_1\cdots g_k K$ (where the $g_j$ are generators of $N$ and $0 \leq k \leq n$).
Let $h_k = g_1 \cdots g_k$ and consider $h_k g_{k+1} \in N$.
Then (since $g_{k+1}$ is a generator of $N$) $h_kg_{k+1} K g_{k+1}^{-1}h_k^{-1}$ has an infinite intersection with $h_kK h_k^{-1}$.
Hence  $h_k g_{k+1} K \cap h_k K g_{k+1}$ is also infinite. 
Consider again a sequence $\gamma_n$ of elements of this intersection. 
Then $f(\gamma_n)$ is constant (since $h_k g_{k+1} K$ is a right-coset of $K$) and the $\gamma_n$ are at distance at most $|g_{k+1}|_S$ from $h_k K$. 
But by induction, $f(h_kK)= f(K)=d$.
Hence $f$ takes the same constant value on $h_kg_{k+1}K$. 
So induction shows that the statement hold for any element of $N$.

This shows that $f$ is constant on $N$. To show that $f$ is constant on any right-coset of $N$, replace $f$ by a translate: let $\lambda_x f(g) := f(x^{-1}g)$. Then $\lambda_xf $ is also constant on the right-cosets of $K$ and $\nabla \lambda_xf \in c_0E$. As a consequence $\lambda_x f$ is constant on $N$, which means that $f$ is constant on $xN$.

Since $f$ is constant on the right-cosets of $N = N_1 = N^q(K)$, then it is also constant on the right-cosets of $N_2 = N^q(N)$. 
This argument can be repeated to any successive ordinal of the sequence $N_{i+1} := N^q(N_i)$. 
If $\kappa$ is a limit ordinal, then the conclusion holds as $N_\kappa = \cup_{i < \kappa} N_i$.
Thus transfinite induction may be applied to show that $f$ is constant on the malnormal hull of $Z$.
\end{proof}

The previous two lemmas could (and in part will) be used as follows. First, one needs to show that a function which is harmonic and has gradient in $c_0$ is also harmonic with respect to some other measure. 
Then conclude that it is constant on the right cosets of some subgroup (here using Lemma \ref{tconcen-l}; but there could be other possibilities).
Finally, use Lemma \ref{tconmar-l} to show it's constant on the whole group. 
The upcoming results are some possibilities for this strategy.

\begin{cor}\label{tdirc0gp-c}
Let $G = G_1 \times G_2$ be a direct product of two infinite (finitely generated) groups.
Then (in any Cayley graph of $G$) if $f$ is harmonic and $\nabla f \in c_0E$, then $f$ is constant.
\end{cor}
\begin{proof}
Indeed, since $G_1 \subset Z(G_2)$ then by Lemma \ref{tconcen-l}, $f$ is constant on the $G_1$-cosets. 
Since one also has $G_2 \subset Z(G_1)$, $f$ is also constant on the $G_2$-cosets. This means that $f$ is constant.
\end{proof}

\begin{cor}\label{tzinf1-c}
Let $G$ be a finitely generated group and let $H < G$ be an infinite subgroup.
Let $Z < Z^{FC}_G(H)$. 
Assume $H \cap Z$ is infinite and $G = \langle  H,Z\rangle$.
Then (in any Cayley graph of $G$) if $f$ is harmonic and $\nabla f \in c_0E$, then $f$ is constant.
\end{cor}
\begin{proof}
If $C:= Z \cap H$ is infinite, then actually $Z^{FC}_G(G)$ is infinite.
Hence, for any generating set $S$, Lemma \ref{tconcen-l} shows that $f$ (which is harmonic on $G$) is constant on $C$.
Since $N^q(C) \supset H$ and $N^q(H) \supset Z$, the almost-malnormal hull of $C$ is $G$ and 
Lemma \ref{tconmar-l} can then be applied to conclude that $f$ is constant on $G$. 
\end{proof}

Upon restricting to a specific Cayley graph, the condition that $H \cap Z$ is infinite can be relaxed.

\begin{teo}\label{tsubcent}
Let $G$ be a finitely generated group and let $H < G$ be an infinite subgroup.
Assume $Z < Z^{FC}_G(H)$ is infinite and $G = \langle H,Z\rangle$.
Then there is a [finite symmetric] generating set $S$, so that, in the associated Cayley graph,
if $f$ is harmonic and $\nabla f \in c_0E$, then $f$ is constant.
\end{teo}
\begin{proof}
Since Corollary \ref{tzinf1-c} deals with the case $H \cap Z$ is infinite, the proof reduces to the case where $C: = Z \cap H$ is finite.

The generating set $S$ which will come into the proof, is any generating can be written as $S = \big( \cup_i C_i \big) \cup \big( \cup_{j,k} D_j h_k \big)$ with $C_i, D_j \subset Z$ being $H$-conjugacy classes and $h_k \in H$. 
To see that such a generating set always exists, note that $Z \lhd G$. 
Hence generating set $S = \{g_i\}_{1 \leq i \leq n}$ can be written as $g_i = z_i h_i$ with $z_i \in Z$ and $h_i \in H$. 
As such $G$ is generated by $C_i$ and $h_i$ where $C_i$ is the $H$-conjugacy class of $z_i$ (in that case $D_j$ is reduced to the identity element).

For such a set, the random walk on the Cayley graph is given by convolution with $P = \sum_i p_i \delta_{z_i} + \big(\sum_{j} q_j \delta_{z_j}\big) \big( \sum_k r_k \delta_{h_k} \big)$ where the $p_i$, $q_j$ and $r_k$ are some positive real numbers (for the simple random walk $p_i = q_j^2 = r_k^2$, but here, it only matters that conjugating the elements of $Z$ under $H$ does not affect the coefficients).
Look at the associated random walk on $H$, given by convolution with $\mu = p \delta_{e} + q\sum_{n} r_n \delta_{h_n}$ where $p = \sum_{i} p_i$ and $q = \sum_{j} q_j$
Then 
\[
\begin{array}{r@{\;}c@{\;}l}
\mu  * P 
&=& p P
+ q  \big(\sum_{n} r_n \delta_{h_n} \big) \big(\sum_i p_i \delta_{z_i} \big)
+ q  \big(\sum_{n} r_n \delta_{h_n} \big) \big(\sum_{j} q_j \delta_{z_j}\big) \big( \sum_k r_k \delta_{h_k} \big) \\
&=& P p
+ \big(\sum_i p_i \delta_{z_i} \big) q \big(\sum_{n} r_n \delta_{h_n} \big) 
+ \big(\sum_{j} q_j \delta_{z_j}\big) \big(\sum_{n} r_n \delta_{h_n} \big) q \big( \sum_k r_k \delta_{h_k} \big)\\
&=& P p + \big(\sum_i p_i \delta_{z_i} \big) q \big(\sum_{n} r_n \delta_{h_n} \big) 
+ \big(\sum_{j} q_j \delta_{z_j}\big) \big(\sum_{k} r_k \delta_{h_k} \big) q \big( \sum_n r_n \delta_{h_n} \big)\\
&=& P *\mu,\\
\end{array}
\]
where the interchange in the second and third terms happen only because the whole conjugacy class of $z_i$ is in the sum. In the second to the third line only a change of indices occur.

From there, one can then consider a harmonic function $f$ on the Cayley graph of $G$ (with generating set $S$).
Then, by harmonicity of $f$ and then the fact that $P*\mu = \mu *P$,  
$
d = \langle f \mid \delta_g - \delta_g *\mu \rangle 
= \langle f \mid \delta_g *P^n - \delta_g * \mu *P^n \rangle 
= \langle f \mid \delta_g *P^n - \delta_g *P^n * \mu \rangle 
$.
There is an obviously bounded transport pattern $\tau_n$ from $P^n$ to $P^n *\mu$ (just do a random step in $H$).
Since $G$ is infinite, $P^n$ (and hence $\tau_n$) tend point-wise to 0. Consequently, $d =0$ which means that $f(g) = \sum_{h} f(gh) \mu(h)$.

This implies that $f$ is harmonic with respect to $\mu$, and $\mu$ generate $H$.
By Lemma \ref{tconcen-l}, $f$ is constant on $Z$.
By Lemma \ref{tconmar-l}, $f$ is constant on the malnormal hull of $Z$, which is $G$.
\end{proof}

The following result applies, among others, to Abelian-by-cyclic groups such as the soluble Baumslag-Solitar groups.

\begin{teo}\label{tc0abcy}
Let $G$ be a [finitely generated] group  which is an cyclic extension $1 \to H \to G \overset{\pi}\to \mathbb{Z} \to 1$.
If $Z^{FC}_H(H)$ is infinite, then any harmonic function on a Cayley graph of $G$ with $c_0$-gradient is constant.
\end{teo}
\begin{proof}
Let $\mu$ denote the random walk on the Cayley graph and $\pi: G \surj \mathbb{Z}$ be the quotient map. Define $\mu_n$ to be the measure defined by firing up $\mu$ at all vertices $g$ such that $\pi(g) \neq 0$ or $|\pi(g)|>n$. In other words, let a random walker start at the identity, and he stops if he hits the set $\pi^{-1}(\{-n,0,n\})$. Then $\mu_n$ is the probability distribution of where he stopped. Write $\mu_n = \beta_n + \zeta_n$ where $\zeta_n$ is the part of $\mu_n$ supported on $\pi^{-1}(0) = H$. By the gambler's ruin, one has that $\zeta_n$ has $1-\tfrac{1}{n}$ of the mass.

Next let $f$ be a harmonic function. Consider $g \in G$ and, for simplicity, assume $c \in Z_H(H)$ (instead of $Z^{FC}_H(H)$). Then, since $f$ is also harmonic with respect to $\mu_n$,
\[
\begin{array}{rll}
f(g) - f(gc) 
&= \langle f | \delta_g *(1- \delta_c) \rangle
&= \langle f | \delta_g *(1- \delta_c) *\mu_n \rangle\\
&= \langle f | \delta_g *(\zeta_n- \delta_c *\zeta_n + \beta_n - \delta_c*\beta_n) \rangle.\\
&= \langle f | \delta_g *(\zeta_n- \zeta_n * \delta_c + \beta_n - \delta_c * \beta_n) \rangle.
\end{array}
\]
where the fact that $c$ commutes with $H$ is used (to commute $\delta_c$ with $\zeta_n$; which has support in $H$).
The transport pattern from $\beta_n$ to $\delta_c * \beta_n$ can use paths of length at most $2n+|c|$ and since the mass of $\beta_n$ is at most $\tfrac{1}{n}$, the $\ell^1$-norm of the transport remains bounded (and it tends to 0 point-wise). Hence $ f(g) - f(gc) = \langle f | \delta_g *(\zeta - \zeta * \delta_c) \rangle$
where $\zeta$ is the limit measure. 

Repeating this argument $k$ times again yield 
$ f(g) - f(gc) = \langle f | \delta_g *(\zeta^k - \zeta^k * \delta_c) \rangle$. 
The transport pattern from $\zeta^k$ to $\zeta^k * \delta_c$ is bounded by the word length of $c$ and tends point-wise to 0.
It follows that $f(g) - f(gc) = 0$.

Consequently, $f$ is constant on the cosets of $Z_H(H)$. Since $N^q( Z_H(H)) \supset H$ and $N^q(H) = G$, Lemma \ref{tconmar-l} shows that $f$ is constant on $G$.

The proof can be adapted to $Z^{FC}_H(H)$ just as in Lemma \ref{tconmar-l}.
\end{proof}
The previous result is similar to a result of Brieussel \& Zheng \cite[Theorem 1.1]{BZ}. In their result $G$ is a [locally normally finite]-by-cyclic group and they show that $G$ has trivial reduced cohomology in degree one for any weakly mixing representations (Shalom's property $H_{FD}$. Such a group $G$ necessarily fulfills the hypothesis of Theorem \ref{tc0abcy}, since a locally finitely normal group is FC-central. On the other hand, Theorem \ref{tc0abcy} only implies the triviality of the reduced cohomology in degree one for strongly mixing representations. As such it seems that both results are not comparable.

Using Theorem \ref{tc0abcy}, as well as \cite[Theorem 5.12]{Go-mixing}, one gets a result on the absence of non-constant harmonic function with gradient in $\ell^p$:
\begin{cor}\label{tlpmetab-c}
Any [finitely generated] metabelian group has $\hdp$ for any $1<p<\infty$.
\end{cor}
\begin{proof}
If $G$ is virtually nilpotent, then the result can be seen as a consequence of Theorem \ref{tinfcen-t} (but there are many possible earlier proofs).

Otherwise if $G$ is not virtually nilpotent, then $G$ contains an Abelian-by-cyclic subgroup $G_0$ of exponential growth (see, for example, Groves \cite{Gr} or Breuillard \cite[Proposition 4.1]{Br}). 
$G_0$ has exponential growth, hence it has $\IS_d$ for any $d$ (see \S{}\ref{ssisop}). Furthermore, by Theorem \ref{tc0abcy}, $G_0$ has $\hdc$, in particular is also has $\hdp$ for any $p \in [1,\infty[$.

Furthermore, if $1 \to A_1 \to G \to A_2 \to 1$ is decomposition of $G$ as an extension of $A_1$ by $A_2$, then note that $N^q(G_0) \supset \langle G_0, A_1 \rangle =: G_1$ (this follows from the fact that the ``Abelian'' subgroup of $G_0$ is infinite and commutes with $A_1$). It follows that $N^q(G_1) \supset N^q(A_1) = G$. So $G$ is the almost-malnormal hull of $G_0$.

By \cite[Theorem 5.12]{Go-mixing} (and the correspondence between $\ell^p$-cohomology and harmonic functions from \cite{Go} or \cite{Go-frontier}) reduced $\ell^p$-cohomology of $G$ is trivial for all $p \geq 1$, which in turn implies that $\hdp$ holds for all $p \geq 1$. 
\end{proof}
\cite[Question 6.2]{Go-mixing} asks whether $\hdp$ holds for any finitely generated soluble groups. 
The previous corollary gives a positive answer for case where the soluble rank is 2.

\begin{rmk}\label{rfinsup}
Let $G$ be a group which is generated by the finite set $S$. 
Let $\mu$ be a measure on $G$.
Given $\alpha >0$, a measure has a finite $\alpha$-moment if $\sum_{g \in G} \mu(g) |g|_S^\alpha < + \infty$ where $|g|_S$ denotes the word length of $g$ 
(\ie the distance between $g$ and the identity in the Cayley graph with respect to $S$).

Let $f$ be a function with $\ell^\infty$-gradient (\ie $f$ is Lipschitz). If $\mu$ has finite 1-moment, then 
$f * \mu (g) := \sum_{h \in G} f(gh) \mu(h)$ will be finite for any $g$.
Indeed, if $\pi_{g,gh}$ denotes a path from $g$ to $gh$ (of length $|h|_S$), one has  
\[
\begin{array}{rcl}
\sum_{h \in G} f(gh) \mu(h) 
&\leq& \displaystyle \sum_{h \in G} \big( f(g) + \sum_{k \in \pi_{g,gh}} \nabla f(k) \big) \mu(h) \\
&\leq& \displaystyle  \sum_{h \in G} \big( f(g) + |h|_S \|\nabla f\|_{\ell^\infty} \big) \mu(h) \\
&\leq& \displaystyle f(g) + \|\nabla f\|_{\ell^\infty} \sum_{h \in G} |h|_S \mu(h).
\end{array}
\]
Hence, for functions $f$ with $\ell^\infty$-gradient, one may consider them harmonic with respect to $\mu$ if $f*\mu = f$.

Note that it is also possible to speak of transport patterns between measures which are not finitely supported.
Indeed, given $f$ with gradient in $\ell^\infty$ as well as two measures with finite first moment $\mu$ and $\nu$, If there is is a $\tau \in \ell^1 E$ so that $\nabla^* \tau = \mu - \nu$, then the equality 
$\langle \nabla f | \tau \rangle =
\langle f | \nabla^* \tau \rangle$ holds (since $f$ is Lipschitz, see above).

This might be useful to prove $\hdc$ for groups.
Indeed, when a group $\Gamma$ has a measure $P$ whose support generate $\Gamma$ and another measure $\mu$ which commutes with $P$ and generates a smaller group, then the methods of this current subsection have allowed us to show that $\Gamma$ has $\hdc$. 
By considering measure which have finite 1-moment, instead of finitely supported, one could be able to prove this property for more groups.

Similar considerations can be made for functions with $\ell^p$-gradient. Indeed, as before
 \[
\begin{array}{rcl}
\sum_{h \in G} f(gh) \mu(h) 
&\leq& \displaystyle \sum_{h \in G} \big( f(g) + \sum_{k \in \pi_{g,gh}} \nabla f(k) \big) \mu(h) \\
\end{array}
\]
Since $\nabla f \in \ell^pE$, one has that $ \big\| \displaystyle \sum_{k \in \pi_{g,gh}} \nabla f(k) \big\|_{\ell^\infty} \leq \|\nabla f\|_{\ell^p} |h|_S^{1/p'}$. Hence the above sum makes sense for measures $\mu$ with finite $\tfrac{1}{p'}$-moment.

%
\end{rmk}

By the upcoming work of Amir, Gerasimova \& Kozma \cite{AGK}, it is however not possible to show $\hdc$ for any metabelian group. Indeed, they show that the lamplighter on $\zz^5$ contain such functions.


\subsection{Divergence and decay of the gradient}\label{ss-div}

The goal of this section is to mention a connection between the divergence (see \eg Ol'shanskii, Osin and Sapir \cite[Definition 1.5]{OOS}) and the decay of the gradient. For our purpose, let the divergence be defined as follows. 

Given some vertex $\mathfr{o}$ and an integer $K > 1$, consider $S_{\mathfr{o},K}(n)$ to be vertices which are in the complement of the ball of radius $n$ around $\mathfr{o}$ but not in the infinite components of the complement of the ball of radius $Kn$. Let $d_S$ be the combinatorial distance on the finite graph induced on $S_{\mathfr{o},K}(n)$ and $S^{out}$ to be the set of vertices which in $S_{\mathfr{o},K}(n)$ and are neighbours to a vertex at distance $Kn+1$ von $\mathfr{o}$.
Then the divergence is $D_{\mathfr{o},K}(n) = \max \{ d_S(x,y) \mid x,y \in S^{out}(n)\}$.

There are many groups where this divergence is a linear function (for any $K$ large enough) such as soluble groups, uniformly amenable groups and groups with a non-trivial law (see Ol'shanskii, Osin and Sapir \cite[\S{}1.3]{OOS}).

Given a function $f$ on the vertices of the graph $G$, let $\mathrm{gd}_{f}(n) = \| \nabla f \|_{\ell^\infty ( E \cap B_{\mathfr{o}}(n)^{\mathsf{c}\infty})}$ be the supremum on edges of the gradient which are in an infinite component of the complement of the ball of radius $n$. 
\begin{cor}
If $h$ is a non-constant harmonic function, then $\displaystyle \lim_{n \to \infty} D_{\mathfr{o},K}(n) \cdot \mathrm{gd}_h (n) \neq 0$. 
\end{cor}
In particular, if $h$ is a non-constant harmonic function on a soluble group then the gradient of $h$ has to decay as slow as $1/n$.
\begin{proof}
Consider two vertices $x$ and $y$ so that $h(x) \neq h(y)$. 
Then, if $n$ is large enough so that $x,y \in B_{\mathfr{o}}(Kn)$, $h(x) - h(y) = \langle \mr{ex}_x^{S^{out}(n)} - \mr{ex}_y^{S^{out}(n)} \mid h \rangle $.
Now let $\tau$ be a transport pattern from $\mr{ex}_x^{S^{out}(n)}$ to $\mr{ex}_y^{S^{out}(n)}$ which avoids the ball of radius $n$.
For this transport pattern, one has an upper bound $\|\tau\|_{\ell^1E} \leq D_{\mf{o},K}(n) \cdot \|\mr{ex}_x^{S^{out}(n)} - \mr{ex}_y^{S^{out}(n)}\|_{\ell^1 X} \leq 2 D_{\mf{o},K}(n)$.
From $h(x) - h(y) = \langle \nabla^*\tau \mid h \rangle = \langle \tau \mid \nabla h \rangle$ and the fact that $\tau$ is supported in $S_{\mf{o},K}(n)$ one gets that 
$h(x) - h(y) \leq \|\tau\|_{\ell^1E} \| \nabla h\|_{\ell^\infty ( E \cap B_{\mathfr{o}}(n)^{\mathsf{c}\infty})} \leq 2 D_{\mathfr{o},K}(n) \cdot \mathrm{gd}_h (n)$.
Consequently the right-hand side is bounded away from 0.
\end{proof}
Note that if the graph is Liouville, then one gets a small improvement since the $\ell^1$-norm of the difference of exit distributions tends to 0.
It would be interesting to apply this result to prove the absence of bounded harmonic functions with gradient in $\ell^p$.

\section{Questions }\label{sques}

%

%

Regarding the first question, 
Here is an example of a graph with infinitely many ends where $\mb{K}_1 = \srl{\mb{k}_1}^*$. 
Consider the half line and 
attach to every vertex another half-line (the vertices of these attached half-line will be labelled by $\nn$). 
It's not too difficult to show that the Dirac mass (in $\ell^1E$) of any edge belongs to $\srl{\mb{k}_1}^*$. 
Indeed, if the edge $(x-1,x)$ belongs to an attached half-line then consider $\nabla f_n$ where $f_n$ is the function supported on $(x, x+n)$.
This $f_n$ tends weak$^*$ to the desired Dirac mass.
Dirac masses which are on the original half-line can then easily be obtained.

Here is an example where $\mb{f}_c \neq \mb{F}_c$. 
Consider an infinite tree without vertices of degree $\leq 2$.
Recall that in a tree $\mb{f}_c$ is trivial.
Hence it suffices to construct an element of $\mb{F}_c$.
Let $\mf{o}_\pm$ be two extremities of some edge. 
One removes this edge and gets two rooted trees $T_\pm$ (with roots $\mf{o}_\pm$).
Draw those trees in a graded fashion (say with the root $\mf{o}_+$ at the top of its tree and the root $\mf{o}_-$ at the bottom of its tree).
Define the value of the function $f$ (on the edges) by setting $f(\vec{e})$ to be $1/ \prod (d_i-1)$ where $d_i$ are the degrees of the vertices between $\vec{e}$ and the root. 
The sign of $f$ should be positive when going down.
Now put back the edge between $\mf{o}_+$ and $\mf{o}_-$ and set the value of $f$ on this edge to be $1$.
It is fairly obvious to see that this function is a flow and tends to $0$ at infinity (\ie belongs to $\mb{F}_c$).

Assume $f \in \ell^1E \setminus \srl{ \mb{K}_1 + \mb{F}_1}^*$. Look at elements $g$ of minimal norm inside $f + \srl{ \mb{K}_1 + \mb{F}_1}^*$. Each $g$ can be used as a weight on the edges and to create a length-metric space $(G,m_g)$ where $m_g(x,y) = \inff{\pi:x \to y} \sum_{\vec{e} \in \pi} |g(\vec{e})|$ where the infimum runs over all paths from $x$ to $y$ and the sum over all edges in the path. Let $\srl{G}^g$ be the completion of $G$ with respect to the metric $m_g$. 

\begin{ques}
Are any of the $\srl{G}^g \setminus G$ Floyd boundaries?
\end{ques}

Given a finite presentation of a group, there is a natural space of cycles as well as an operator $\nabla_2: \ell^pC \to \ell^pE$. 
The fact that $\nabla_2$ has closed image for $p=1$ seems to be related to hyperbolicity. 
A more general question would be:
\begin{ques}\label{qnabla2}
Does the fact that $\nabla_2$ has closed image in $\ell^2$ is related to a group-theoretic property?
\end{ques}

For Cayley graphs, the property $\hdp$ (absence of non-constant harmonic functions with gradient in $\ell^p$) is very close to be invariant under quasi-isometry. 
Indeed, in \cite{Go}, it is shown that if a graph has $\hdp$ then its reduced $\ell^p$-cohomology is trivial and that this in turn implies $\hdq$ for any $q<p$.
Since the triviality of the reduced $\ell^p$-cohomology is an invariant of quasi-isometry, one gets that, for any fixed $p$, ``having $\hdq$ for any $q<p$'' is also an invariant of quasi-isometry. 
However it remains open to check that:
\begin{ques}\label{qisom}
Is $\hdp$ an invariant of quasi-isometry (for Cayley graphs)?
\end{ques}
In the case of generic graphs, some trouble could come up in the class of graphs which do not have $\IS_d$ for some $d$. 
The question could also be asked for graphs which have $\IS_d$ for all $d$ (instead of Cayley graphs). The quasi-isometry invariance of $\hdc$ (absence of non-constant harmonic functions with gradient in $c_0$) is also open. In this case the following is also unclear:
\begin{ques}\label{qbdd}
Is $\hdc$ equivalent to the absence of non-constant \emph{bounded} harmonic function with gradient in $c_0$?
\end{ques}
In some sense, the quasi-isometry invariance could be as hard as the quasi-isometry invariance of the Liouville property. However, note that, contrary to the space of bounded harmonic functions, the space of harmonic functions with gradient in $\ell^p$ (resp. $c_0$) is closed in $\ell^pE$ (resp. $c_0E$).

In \cite{ER}, Elder \& Rogers investigate how to recover the amenable radical using some properties on random walks.
Let $H \subset G$ be the set of all elements $h$ so that 
\[
\| \mr{ex}^{A_n}_h - \mr{ex}^{A_n}_\mf{o} \|_{\ell^1} \to 0
\]
The triangle inequality combined with group invariance shows this is actually a subgroup of $G$.
This may not be the amenable radical simply because it will not equal $G$ in any group which is not Liouville.
For example, on the lamplighter on $\zz^3$, $G = C_2 \wr \zz^3$, $H \lneq G$ (it seems that $H = \{1\}$).
Also, if $G = G_1 \times G_2$ where $G_1$ is Liouville and $G_2$ is not, then $H=G_1$.

\begin{ques}
What are the properties of this subgroup $H<G$? 
\end{ques}

In a Cayley graph, a special case [restricting to degree $1$] of a question (dating back at least to Gromov \cite[\S{}8.$A_1$.$A_2$, p.226]{Gro}) can be formulated (\emph{via} the results of \cite{Go} and Lemma \ref{tannilp-l}) as 
\begin{ques}
Is it true that $\mb{f}_p + \mb{K}_p$ is dense in $\ell^pE$ for the Cayley graph of any amenable group and all $p \in ]1,2[$?
\end{ques}

The main objective in proof of Theorem \ref{thdp-t}, is that one notes immediately that only part in the proof where we use the Liouville condition is to show $\tau_k \overset{p.-w.}\rightharpoonup 0$.
To develop this further, say that a transport pattern which only uses geodesic will be called \textbf{optimal}.
A transport pattern is $K$-optimal (for some $K \geq 1$) if the mass is transported by at most $K$ times the minimal length (the distance it would travel in the optimal transport).
Note that if $\tau''$ is $K$-optimal and $\tau'$ is optimal, then $\|\tau''\|_{\ell^1} \leq K \|\tau'\|_{\ell^1}$.

Let $\mu_k = P^k \delta_\mf{o} - P^k \delta_s$. 
Take a sequence $\tau_k$ such $\nabla^* \tau_k = \mu_k$ so that $\tau_k \overset{p.-w.}\rightharpoonup 0$. 
Let $K_k$ be the smallest real number such that $\tau_k$ is $k$-optimal.
For each sequence the function $k \mapsto K_k$ gives some kind of divergence as measured by the random walk.

\begin{ques}
Are there groups for which the divergence is linear, but there is a sequence $\tau_k$ as above so that $k \mapsto K_k$ is bounded?
\end{ques}
Although the existence of such a sequence is sufficient for $\hdp$, it might not be necessary. Still, this gives a good hint at which amenable groups might not have $\hdp$.
\begin{ques}
Is there an amenable lacunary hyperbolic group with $\hdp$? 
\end{ques}
More precisely, does the group introduced in Ol'shanskii, Osin \& Sapir \cite[\S{}3.5]{OOS} has $\hdp$? Note that it is not known whether this group has the Liouville property (no non-constant bounded functions).

\end{document}